\documentclass{article}
\usepackage[a4paper, total={7in, 10in}]{geometry}
\usepackage[T1]{fontenc}
\usepackage{graphicx}
\usepackage{mathtools}
\usepackage{amsmath,amsthm,amssymb,mathrsfs,amstext, titlesec,enumitem, comment, graphicx, color, xcolor, stmaryrd,mathabx}
\usepackage{thmtools}
\usepackage{xpatch}
\usepackage{tikz-cd}
\usepackage{nameref}
\usepackage{hyperref}
\makeatletter
\def\Ddots{\mathinner{\mkern1mu\raise\p@
\vbox{\kern7\p@\hbox{.}}\mkern2mu
\raise4\p@\hbox{.}\mkern2mu\raise7\p@\hbox{.}\mkern1mu}}
\makeatother
\newtheorem{theorem}{Theorem}[section]

\newtheorem{lemma}[theorem]{Lemma}
\newtheorem{question}[theorem]{Question}
\theoremstyle{definition}

\newtheorem{definition}[theorem]{Definition}

\begin{document}
\title{\textbf{On dynamical $C^{\star}$-set and its combinatorial consequences}}

	\date{}
	\author{Pintu Debnath
		\footnote{Department of Mathematics,
			Basirhat College,
			Basirhat-743412, North 24th parganas, West Bengal, India.\hfill\break
			{\tt pintumath1989@gmail.com}}
		\and
		 Sayan Goswami 
		\footnote{Department  Mathematics,
			Ramakrishna Mission Vivekananda Educational and Research Institute,
			Belur, Howrah, 711202, India\hfill\break
			{\tt sayan92m@gmail.com}	
	}
}
\maketitle	
\begin{abstract}
  Using the methods from topological dynamics,  H. Furstenberg introduced the notion of a central set and proved the famous Central Sets Theorem. Later D. De, Neil Hindman, and D. Strauss [Fund. Math.199 (2008), 155-175.] established a stronger version of the Central Sets Theorem and then introduced the notion of $C$-sets satisfying the Central Sets Theorem and studied the properties of these sets. For any weak mixing system $\left(X, \mathcal{B},\mu, T\right),$ and $A_{0},A_{1}\in\mathcal{B}$, with $\mu\left(A_{0}\right)\mu\left(A_{1}\right)>0$,  R. Kung and X.Ye [Disc. Cont. Dyn. sys., 18 (2007) 817-827.]  proved that the set $N\left(A,B\right)= \left\{n:\mu\left(A_{0}\cap T^{-n}A_{1}\right)>0\right\}$ intersects all sets of positive upper Banach density. However, later  N. Hindman and D. Strauss [New York J. Math. 26 (2020) 230-260.] proved that there exist $C$-sets having zero upper Banach density. Inspired by this result, in this article, we prove that  $N\left(A, B \right)$ intersects with all $C$-sets. Then we introduce the notion of a dynamical $C^{\star}$-set and then we study their combinatorial properties.
\end{abstract}

	

\section{Introduction}


Let $X$ be a compact Hausdorff space and $T:X\rightarrow X$ is a continuous map. We call $(X,T)$ is a dynamical system. For any open set $U\subseteq X,$ and $x\in X$ the set $R(x,U)=\{n:T^nx\in U\}$ is called the return times set. Instead of $\mathbb{Z}$ action, one can easily define the dynamical systems for any semigroup action.  In \cite[Proposition 8.21]{F81}, using return times set,  H. Furstenberg introduced the notions of Central sets and proved the famous Central Sets Theorem. Before we proceed let us assume that the set $\mathbb{N}$ is the set of all positive integers, and for any nonempty set $X,$ let $\mathcal{P}_{f}(X)$ be the set of all nonempty finite subsets of $X.$ By $X^\mathbb{N},$ let us denote the set of all sequences over $X.$

\begin{theorem}
    \label{cst} \textup{[\textbf{Central Sets Theorem}]}
Let $l\in\mathbb{N}$, and $A\subseteq\mathbb{N}$ be a central set. For
each $i\in\{1,2,\ldots,l\}$ let $\langle x_{i,m}\rangle_{m=1}^{\infty}$ be a sequence in $\mathbb{\mathbb{N}}$. Then there exists a sequence
$\langle b_{m}\rangle_{m=1}^{\infty}$ in $\mathbb{N}$ and $\langle K_{m}\rangle_{m=1}^{\infty}$
in $\mathcal{P}_{f}\left(\mathbb{N}\right)$ such that
\begin{enumerate}
\item For each $m$, $\max K_{m}<\min K_{m+1}$ and
\item For each $i\in\{1,2,\ldots,l\}$ and $H\in\mathcal{P}_{f}\left(\mathbb{N}\right)$,
$\sum_{m\in H}(b_{m}+\sum_{t\in K_{m}}x_{i,t})\in A$.
\end{enumerate}
\end{theorem}

Later using the methods of Stone-\v{C}ech compactification of discrete semigroups, in \cite{DHS08}, D. De, N. Hindman, and D. Strauss proved the following stronger version of the Central Sets Theorem. 

\begin{theorem}
    \label{gcst} \textup{[\textbf{Stronger Central Sets Theorem}]} Let $(S,+)$ be any commutative semigroup, $\tau=S^{\mathbb{N}}$ and let $C \subseteq S$ be central. There exists functions $\alpha : \mathcal{P}_{f}(\tau)\to S$ and $H: \mathcal{P}_{f}(\tau) \to \mathcal{P}_{f}\left(\mathbb{N}\right)$ such that 
\begin{enumerate}
\item \label{1.41} if $F,G \in \mathcal{P}_{f}(\tau)$ and $F \subsetneqq G$ then $\max H(F) < \min H(G)$, and 
\item \label{1.42} whenever $m \in \mathbb{N}$, $G_1, G_2, \ldots , G_m \in \mathcal{P}_{f}(\tau)$, $G_1 \subsetneq G_2 \subsetneq \cdots \subsetneq G_m$ and for each $i \in \{1,2, \ldots , m\}$, $f_i \in G_i$, one has $$\sum_{i=1}^{m}\big(\alpha(G_i)+\sum_{t\in H(G_i)}f_i(t)\big)\in C.$$
\end{enumerate}
\end{theorem}
In \cite{DHS08}, authors introduced the notions of $C$-sets: sets which satisfy the conclusion of the Theorem \ref{gcst} are $C$ sets. In this article, we will study the compatibility of $C$ sets and measure preserving systems. Before we proceed, we need to recall the basic preliminaries of the Stone-\v{C}ech compactification of discrete semigroups.

Let $S$ be a discrete semigroup. The elements of $\beta S$ are regarded as ultrafilters on $S$. Let $\overline{A}=\left\{p\in \beta S:A\in p\right\}$. The set $\{\overline{A}:A\subset S\}$ is a basis for the closed sets
of $\beta S$. The operation `$\cdot$' on $S$ can be extended to
the Stone-\v{C}ech compactification $\beta S$ of $S$ so that $(\beta S,\cdot)$
is a compact right topological semigroup (meaning that for each    $p\in\beta S$  the function $\rho_{p}\left(q\right):\beta S\rightarrow\beta S$ defined by $\rho_{p}\left(q\right)=q\cdot p$ 
is continuous) with $S$ contained in its topological center (meaning
that for any $x\in S$, the function $\lambda_{x}:\beta S\rightarrow\beta S$
defined by $\lambda_{x}(q)=x\cdot q$ is continuous). This is a famous
theorem due to Ellis that if $S$ is a compact right topological semigroup
then the set of idempotents $E\left(S\right)\neq\emptyset$. A nonempty
subset $I$ of a semigroup $T$ is called a $\textit{left ideal}$
of $S$ if $TI\subset I$, a $\textit{right ideal}$ if $IT\subset I$,
and a $\textit{two sided ideal}$ (or simply an $\textit{ideal}$)
if it is both a left and a right ideals. A minimal left ideal
is a left ideal that does not contain any proper left ideal. Similarly,
we can define a minimal right ideal and the smallest ideal.

Any compact Hausdorff right topological semigroup $T$ has the smallest
two sided ideal

$$
\begin{aligned}
	K(T) & =  \bigcup\{L:L\text{ is a minimal left ideal of }T\}\\
	&=  \bigcup\{R:R\text{ is a minimal right ideal of }T\}.
\end{aligned}$$

Given a minimal left ideal $L$ and a minimal right ideal $R$, $L\cap R$
is a group, and in particular contains an idempotent. If $p$ and
$q$ are idempotents in $T$; we write $p\leq q$ if and only if $pq=qp=p$.
An idempotent is minimal with respect to this relation if and only
if it is a member of the smallest ideal $K(T)$ of $T$. Given $p,q\in\beta S$
and $A\subseteq S$, $A\in p\cdot q$, if and only if the set $\{x\in S:x^{-1}A\in q\}\in p$,
where $x^{-1}A=\{y\in S:x\cdot y\in A\}$. See \cite{HS12} for
an elementary introduction to the algebra of $\beta S$ and any
unfamiliar details.
Before we proceed let us recall the notions of $IP$ sets and $J$ sets. 
\begin{definition}
	Let $\left(S,+\right)$ be a commutative semigroup.
	\begin{enumerate}
		\item A set $A\subseteq S$ is said to be an $IP$-set if and only if  there exists a sequence $\langle x_{n}\rangle _{n=1}^{\infty}$ in $S$ such that $$ FS\left(\langle x_{n}\rangle _{n=1}^{\infty}\right)=\left\{ \sum_{n\in F}x_{n}:F\in\mathcal{P}_{f}\left(\mathbb{N}\right)\right\}\subseteq A. $$
		
		\item A set $A\subseteq S$ is said to be a $J$-set if  for every $F\in\mathcal{P}_{f}\left({}^\mathbb{N}S\right)$, there exist $a\in S$ and $H\in\mathcal{P}_{f}\left(\mathbb{N}\right)$ such that for each $f\in F$, $$a+\sum_{n\in H}f(n)\in A\,.$$
		
	\end{enumerate}
\end{definition}

Let $J\left(S\right)=\left\{p\in\beta S:\left(\forall\in p\right) \left(A\text{ is a } J\text{ set }\right) \right\}$. By \cite[Lemma 14.14.5]{HS12}, and \cite[Theorem 3.20]{HS12} $J\left(S\right)\neq\emptyset$. It was shown in \cite[Theorem 14.14.4]{HS12} that the set $J\left(S\right)$  is a compact two-sided ideal of $\beta S$. 
Now we recall the notions of Central sets and $C$ sets in terms of the ultrafilters.
\begin{definition}
    Let $\left(S,+\right)$ be a commutative semigroup and let $A\subseteq S$.
    \begin{itemize}
        \item[(a)](\textbf{Central set}) $A$ is a central set if and only if $A\in p$ for some idempotent $p$ in $\overline{A}\cap K\left(\beta S\right)$.
        \item[(b)](\textbf{$C$-set}) $A$ is a $C$ set if and only if $A\in p$ for some idempotent $p$ in $\overline{A}\cap J\left(S\right)$.
    \end{itemize}
\end{definition}

In our work, we need the notions of measure-preserving systems. Here we recall the notion.

\begin{definition}{\cite[Definition 19.29]{HS12}}
	\begin{itemize}
		\item[(a)]  A measure space is a triple $\left(X,\mathcal{B},\mu\right)$ where $X$ is a set, $\mathcal{B}$ is a $\sigma$-algebra of subsets of $X$, and is a countably additive measure on $\mathcal{B}$ with $\mu\left(X\right) $ finite.
		\item[(b)] Given a measure space $\left(X,\mathcal{B},\mu\right) $ a function $ T:X\rightarrow X$ is a measure preserving transformation if and only if for all $B\in\mathcal{B}$, $ T^{-1}B\in\mathcal{B}$ and $\mu\left(T^{-1}B\right)=\mu\left(B\right)$
		\item[(c)] Given a semigroup $S$ and a measure space $\left(X,\mathcal{B},\mu\right)$ a measure preserving action of $S$ on $X$ is an indexed family $ \langle T_{s}\rangle_{s\in S}$ such that each $T_{s}$ is a measure preserving transformation of $X$ and $T_{s}\circ T_{t}=T_{st}$ . It is also required that if $ S$ has an identity$ e$ then $ T_{e} $ is the identity function on $X$.
		\item[(d)] A measure preserving system is a quadruple $ \left(X,\mathcal{B},\mu,\langle T_{s}\rangle_{s\in S}\right) $ such that $ \left(X,\mathcal{B},\mu\right) $ is a measure space and $ \langle T_{s}\rangle_{s\in S}$ is a measure preserving action of $ S$ on $ X $.
	\end{itemize}
	
\end{definition}	
 In this article we consider, $\mu\left(X\right)=1$, i.e., probability measure,  $\left(S,\cdot\right)	= \left(\mathbb{N},+\right)$ with $T_{n}=T^{n} $ and $T_{n}^{-1}=T^{-n}$.

The notions of strong mixing, mild mixing, and weak mixing in Ergodic Ramsey theory is a very useful tools. We recall these notions in terms of the largeness of the recurrence of sets. Note that the notions of mild mixing lies between weak and strong mixing. For various equivalent characterizations of these mixings, one can follow the article \cite{KY07} of R. Kung and X. Ye. But for our sake we use the following definition.
\begin{definition}
    Let $\left(X, \mathcal{B}, \mu, T\right)$ be a measure preserving system and let for any $\epsilon>0$, $A,B\in\mathcal{B}$,  $N_{\epsilon}\left(A,B\right)=\left\{n:|\mu\left(A\cap T^{-n}B\right)-\mu\left(A\right)\mu\left(B\right)|<\epsilon\right\}$. Then 

\begin{itemize}
    \item[(a)] $\left(X, \mathcal{B}, \mu, T\right)$ is strong mixing iff for any $\epsilon>0$, $A,B\in\mathcal{B}$, $N_{\epsilon}\left(A,B\right)$ is co-finite.
    \item[(b)]\label{mild mixing} $\left(X, \mathcal{B}, \mu, T\right)$ is mild mixing iff for any $\epsilon>0$, $A,B\in\mathcal{B}$, $N_{\epsilon}\left(A,B\right)$ is an $IP^{\star}$-set.
    \item[(c)]\label{weak mixing} $\left(X, \mathcal{B}, \mu, T\right)$ is weak mixing iff for any $\epsilon>0$, $A,B\in\mathcal{B}$, $N_{\epsilon}\left(A,B\right)$ is central$^{\star}$-set.
  
\end{itemize}
\end{definition}

We get a definition of mild mixing 
\ref{mild mixing} from  \cite[Propoposition 9.22]{F81} and definition of weak mixing from a simple consequence of \cite[Theorem 8]{Mo}. 

For any $A\subseteq \mathbb{N},$ the upper Banach  density of $A$ is defined as $d^\star\left(A\right)=\limsup_{(n-m)\rightarrow\infty}\frac{\mid A\cap\left\{ n+1,n+2,\ldots,m\right\} \mid}{n-m}.$
In \cite[Theorem 4.7]{KY07}, R. Kung and X. Ye  proved that a measure-preserving system $\left(X, \mathcal{B}, \mu, T\right)$ is weak mixing if and only if for any $A_{0},A_{1}\in\mathcal{B}$, with $\mu\left(A_{0}\right)\mu\left(A_{1}\right)>0$, the set $N\left(A_0,A_1\right)= \left\{n:\mu\left(A_{0}\cap T^{-n}A_{1}\right)>0\right\}$ is of upper Banach density one. Now we know that there is a $C$-set with zero upper Banach density in \cite{HS09}. Despite that in the next section we will show that for any weak mixing system $\left(X, \mathcal{B}, \mu, T\right)$ and $A_{0},A_{1}\in\mathcal{B}$, with $\mu\left(A_{0}\right)\mu\left(A_{1}\right)>0$, the set $N\left(A_0,A_1\right)$ is a $C^{\star}$-set. We prove the following stronger version.

\begin{theorem}\label{weakc*}
    Let $\left(X, \mathcal{B}, \mu, T\right)$ be a measure preserving system. Then $\left(X, \mathcal{B}, \mu, T\right)$ is weak mixing iff for any $\epsilon>0$, $A,B\in\mathcal{B}$, $N_{\epsilon}\left(A,B\right)$ is a $C^{\star}$-set.
\end{theorem}
As $C^\star$ sets are Central$^\star$ sets, the converse part of the above theorem immediately follows by the definition \ref{weak mixing}.
Now we introduce the notions of Dynamical  $C^{\star}$- sets.

\begin{definition}[\textbf{Dynamical  $C^{\star}$- sets}]\label{D}
	 	A subset $A$ of $\mathbb{N}$ is a  dynamical  $C^{\star}$- set if there exist a weak mixing system $\left(X,\mathcal{B},\mu,T\right)$ and  $A_{0},A_{1}\in\mathcal{B}$ with $\mu\left(A_{0}\right)\mu\left(A_{0}\right)>0$ such that
   $$\left\{ n\in\mathbb{N}:\mu\left(A_{0}\cap T^{-n}A_{0}\right)>0\right\} \subseteq A.$$	
	 \end{definition}

In Section $3,$ we prove that Dynamical  $C^{\star}$- sets contain rich combinatorial configurations. But to achieve our goal, we need to recall the following definitions of sum subsystems.

\begin{definition} \textbf{(Sum subsystem)}\label{sumsubsystem}
    For any $IP$ set  $FS\left(\langle x_{n}\rangle_{n}\right)$, a sum subsystem of  $FS\left(\langle x_{n}\rangle_{n}\right)$ is a set of the form $FS\left(\langle y_{n}\rangle_{n}\right)$, where for each $n\in \mathbb{N},$ $y_n$ is defined as follows:

\begin{enumerate}
 \item there exists a sequence $\langle H_n\rangle_n$  in $\mathcal{P}_f(\mathbb{N})$ such that for each $n\in \mathbb{N}$, $\max H_n<\min H_{n+1}$, and 
    \item $y_n=\sum_{t\in H_n}x_t$ for all $n\in \mathbb{N}$.
\end{enumerate}
\end{definition}

Any $IP^\star$-set contains a great amount of combinatorial rich patterns. One of the first such results was proved in \cite{BH94}.

\begin{theorem}\label{ipstar}
	Let $\langle x_{n}\rangle_{n=1}^{\infty}$ be a sequence in $\mathbb{N}$
	and $A$ be an IP$^{\star}$-set in $\left(\mathbb{N},+\right)$. Then
	there exists a subsystem $FS\left(\langle y_{n}\rangle_{n=1}^{\infty}\right)$ of
	$FS\left(\langle x_{n}\rangle_{n=1}^{\infty}\right)$ such that $FS\left(\langle y_{n}\rangle_{n=1}^{\infty}\right)\cup FP\left(\langle y_{n}\rangle_{n=1}^{\infty}\right)\subseteq A$.
\end{theorem}
Later in \cite{G21}, S.Goswami  proved that dynamical $IP^\star$ sets contain richer patterns than the conclusion of Theorem \ref{ipstar}. To address that result we need the following definition. 

\begin{definition}[\textbf{Zigzag sum-product}]
	Let $l\in\mathbb{N}$, and $\langle x_{n}^{\left(1\right)}\rangle_{n=1}^{\infty},\langle x_{n}^{\left(2\right)}\rangle_{n=1}^{\infty},\ldots,\langle x_{n}^{\left(l\right)}\rangle_{n=1}^{\infty}$ in $\mathbb{N}$ be $l$-sequences. Define
	\begin{itemize}
		\item[(a)] $ZFS\left(\langle x_{n}^{\left(i\right)}\rangle_{i,n=1,1}^{l,\infty}\right)$
		$=\left\{\sum_{t\in H}y_{t}:H\in\mathcal{P}_{f}\left(\mathbb{N}\right) \text{ and } y_{i}\in\left\{ x_{i}^{\left(1\right)},x_{i}^{\left(2\right)},\ldots,x_{i}^{\left(l\right)}\right\}	\text{ for any } i\in \mathbb{N}\right\}$.
  
		\item[(b)] $ZFP\left(\langle x_{n}^{\left(i\right)}\rangle_{i,n=1,1}^{l,\infty}\right)$
		$=\left\{\prod_{t\in H}y_{t}:H\in\mathcal{P}_{f}\left(\mathbb{N}\right) \text{ and } y_{i}\in\left\{ x_{i}^{\left(1\right)},x_{i}^{\left(2\right)},\ldots,x_{i}^{\left(l\right)}\right\}	\text{ for any } i\in \mathbb{N}\right\}$.
	\end{itemize}
\end{definition}

The following theorem was proved in \cite{G21}.

\begin{theorem}\label{zigzagipstar}
	Let $l\in\mathbb{N}$ and $A\subseteq\mathbb{N}$ be a dynamical $IP^{\star}$ set in $\left(\mathbb{N},+\right)$. Then for any $l$ sequences $\langle x_{n}^{\left(1\right)}\rangle_{n=1}^{\infty}$,$\ldots$,$\langle x_{n}^{\left(l\right)}\rangle_{n=1}^{\infty}$ in $\mathbb{N}$, the followings is true:
 \begin{enumerate}
     \item[(*)] for each $i\in\left\{ 1,2,\ldots,l\right\},$ there exists sum subsystems $\langle y_{n}^{\left(i\right)}\rangle_{n=1}^{\infty}$ of  $\langle x_{n}^{\left(i\right)}\rangle_{n=1}^{\infty}$   such that $$ZFS\left(\langle y_{n}^{\left(i\right)}\rangle_{i,n=1,1}^{l,\infty}\right)\bigcup ZFP\left(\langle y_{n}^{\left(i\right)}\rangle_{i,n=1,1}^{l,\infty}\right)\subset A.$$
 \end{enumerate}
\end{theorem}

An immediate question appears if we can weaken the hypothesis of Theorem \ref{ipstar}. Adapting the similar argument of Theorem \ref{ipstar}, in, D. De proved a version of Theorem \ref{ipstar} for central$^{\star}$
sets. A sequence $\langle x_{n}\rangle_{n=1}^{\infty}$ is called  minimal sequence if  $FS\left(\langle x_{n}\rangle_{n=1}^{\infty}\right)$ is a central set. In \cite{D07} D. De proved if $A$ is a central$^{\star}$ set, then the conclusion of Theorem \ref{ipstar} is true for any minimal sequence. Similarly a sequence $\langle x_{n}\rangle_{n=1}^{\infty}$ is called almost  minimal sequence if  $FS\left(\langle x_{n}\rangle_{n=1}^{\infty}\right)$ is a $C$ set. In \cite{D14}, D. De proved if $A$ is a $C^{\star}$ set, then the conclusion of Theorem \ref{ipstar} is true for any almost minimal sequence.

\section{Dynamical $C^{\star}$-set}

Compact vectors are an essential tool for studying measure-preserving systems. These vectors characterize weak mixing systems.

\begin{definition}[Compact vector]\label{compact vector}
    Let $\left(X,\mathcal{B},\mu,T\right)$ be a measure preserving system. A vector $f\in L^{2}\left(\mu\right)$ is compact if $\overline{O}\left(f\right)=\overline{\left\{T^{n}f:n\in \mathbb{N}\right\}}$ is compact in the norm topology.
\end{definition}

In \cite[Theorem 3(7)]{Mo}, an equivalence relation between compact vector and weak mixing systems was established.

\begin{theorem}
     A measure preserving system $\left(X,\mathcal{B},\mu,T\right)$ is weak mixing if there is no nonconstant compact vector $f\in L^{2}\left(\mu\right)$.
\end{theorem}

It is well known that if $D$ is a set of positive upper Banach densities, then $D-D$ is syndetic.
 However, the upper Banach density of a $J$-set may be zero, although the following lemma shows that the difference set is syndetic. And this fact plays the main role in the proof of the Theorem \ref{weakc*}.

\begin{lemma}\label{J-J}
	If  $A\subseteq \mathbb{N}$ be a $J$-set, then 
	that $A-A$ is an $IP^{\star}$-set in $\mathbb{N}$.
\end{lemma}

\begin{proof}

	Let $\langle x_{n}\rangle _{n=1}^{\infty}$ be any sequence in $\mathbb{N}$.
	Choose two sequences $\langle y_{n}\rangle _{n=1}^{\infty}$
	and $\langle x_{n}+y_{n}\rangle _{n=1}^{\infty}$ in $\mathbb{N}$. Define two functions
	$f,g:\mathbb{N}\rightarrow \mathbb{N}$ by $f\left(n\right)=y_{n}$
	and $g\left(n\right)=x_{n}+y_{n}$ for all $n\in\mathbb{N}$. Then
	there exist $a\in \mathbb{N}$ and $H\in\mathcal{P}_{f}\left(\mathbb{N}\right)$,
	such that $a+\sum_{n\in H}x_{n}\in A$ and $a+\sum_{n\in H}(x_{n}+y_{n})\in A$ which implies $\sum_{n\in H}x_{n}\in A-A$. Hence $A-A$ is an $IP^{\star}$-set.
\end{proof}

One can consider limits along ultrafilters in any compact Hausdorff space. This is one of the most useful techniques in Ergodic Ramsey theory.
\begin{definition}{\cite[Definition 3.44]{HS12}}\label{P-limit}
		Let $D$ be a discrete space, let $p\in \beta D$, let $\langle x_{s}\rangle_{s\in D}$ be an indexed family in a topological space $X$, and let $y\in X$. Then $p\text{-}{\lim_{s\in D}} x_{s}=y$ if and only if for every neighborhood $U$ of $y$, $\left\{s\in D:x_{s}\in U\right\}\in p$. 
	\end{definition}
Let us recall the following properties of $p-\lim $ that we use.
	\begin{theorem}\textup{\cite[Theorem 3.48]{HS12}}
		Let $D$ be a discrete space, let $p\in \beta D$, and let $\langle x_{s}\rangle_{s\in D}$ be an index family in a topological space $X$.
		\begin{itemize}
		\item[(a)]If $p\text{-}{\lim_{s\in D}} x_{s}$ exists, then it is unique.
		\item[(b)] If $X$ is a compact space, then $p\text{-}{\lim_{s\in D}} x_{s}$ exists.
		\end{itemize}
	\end{theorem}
The following theorem follows immediately from the definition of $p-\lim .$

\begin{theorem}\label{vd}\textup{\cite[Theorem 4.5]{HS12}}
	Let $\left(S,\cdot\right)$ be a semigroup, $X$ be a topological space, $\langle x_{s}\rangle_{s\in S}$  be an index family in $X$, and $p,q\in \beta S$. If all limits involved exists, then $$\left(pq\right)\text{-}\lim_{v\in S}x_{v}=p\text{-}\lim_{s\in S}q\text{-}\lim_{t\in S}x_{st}.$$
\end{theorem}

The following lemma gives the relation between compact vectors and idempotents in $J\left(\mathbb{N}\right)$.

\begin{lemma}\label{CI}
     For idempotent $p$ in $J\left(\mathbb{N}\right)$,  $p\text{-}\lim_{n}T^{n}f=f$ implies $f$ is a compact vector.
\end{lemma}

\begin{proof}
Fix $\epsilon>0,$ and let
$$E=\left\{n:\parallel{T^{n}f-f}\parallel<\frac{\epsilon}{4}\right\}\in p.$$
Let $n=n_{1}-n_{2}\in E-E$ with $n_{1}>n_{2}$. Now
$$\begin{aligned}
    \parallel{T^{n}f-f}\parallel &=\parallel{T^{n_{1}}f-T^{n_{2}}f}\parallel\\
    & \leq \parallel{T^{n_{1}}f-f}\parallel+\parallel{T^{n_{2}}f-f}\parallel\\
    & < \frac{\epsilon}{2}.
\end{aligned}$$

So diam$\left(\left\{T^{n}f:n\in E-E\right\}\right)<\epsilon$. Let for some $k\in\mathbb{N}$ and $n_{i}\in\mathbb{N}$ for all $i\in\left\{1,2,\ldots,k\right\}$ such that $\mathbb{N}=\bigcup_{i=1}^{k}\left(-n_{i}+A\right)$.  Let $V=\left\{T^{n}f:n\in E-E\right\}$.
Let $V_{i}=\left\{T^{n}f:n\in -n_{i}+E-E\right\}=T^{-n_{i}}\left(V\right)$.  Then for each $i\in\left\{1,2,\ldots,n\right\}$ one has $diam\left(V_{i}\right)<\epsilon$ and $\left\{T^{n}f:n\in \mathbb{N}\right\}=\bigcup_{i=1}^{k}V_{i}$. Since $\epsilon>0$ is arbitrary, we conclude that $\overline{\left\{T^{n}f:n\in \mathbb{N}\right\}}$ is compact.
\end{proof}

Now we are in the position to prove Theorem \ref{weakc*}.

\begin{proof}[\textbf{Proof of theorem \ref{weakc*}}.]
Let $p\in E \left(J\left(\mathbb{N}\right)\right)$ and  $f\in L^{2}\left(\mu\right)$ with $ \int fd\mu=0$. Now, from \ref{vd} $$g=p\text{-}\lim_{n}T^{n}f\implies p\text{-}\lim_{n}T^{n}g=g$$ and so by \ref{CI}, $g$ is a compact vector.  As $\left(X,\mathcal{B},\mu,T\right)$ is weak mixing, we have $g=0$. Then $p\text{-}\lim_{n}T^{n}f=0$ weakly. Now choose $f_{1}=1_{A}-\mu\left(A\right)$ and $f_{2}=1_{B}.$ Then $p\text{-}\lim_{n}\langle T^{n}f_{1},f_{2}\rangle=0$ which implies $$p\text{-}\lim_{n}\mu\left(A\cap T^{-n}B\right)=\mu\left(A\right)\mu\left(B\right).$$
\end{proof}

As the upper Banach density of any dynamical $C^{\star}$-sets is one, we have any $C^{\star}$-sets with upper Banach density less than one is not a $C^{\star}$-set.  N. Hindman and D. Strauss found an $C$-set, $A\subset \mathbb{N}$ with zero upper Banach density. So, $\mathbb{N}\setminus A$ is of upper Banach density one but not a $C^{\star}$-set.

As we know from the definition of $C^{\star}$-sets, that a $C^{\star}$-set is a member of all idempotents of $J\left(\mathbb{N}\right)$. The following, analog of \cite[Theorem 4.13]{BH12}  shows that a dynamical $C^{\star}$-set is a member of sums of finitely many idempotents in $J\left(\mathbb{N}\right)$.

  \begin{theorem}
Let $k\in\mathbb{N}$.	Let $C$ be a dynamical $C^{\star}$-set in $\mathbb{N}$. If $p_{1},p_{2},\ldots,p_{k}\in E\left(J\left( \mathbb{N}\right)\right)$, then $C\in p_{1}+p_{2}+\ldots+p_{k}$.
\end{theorem}	

\begin{proof}
   	As $C$ is a dynamical  $C^{\star}$-set in $\mathbb{N}$, there exist a weak mixing system $\left(X,\mathcal{B},\mu,T\right)$  and $A_{0} , A_{1}\in\mathcal{B}$ with $\mu\left(A_{0}\right)\mu\left(A_{1}\right)>0$ such that
    $$B=\left\{ n\in\mathbb{N}:\mu\left(A_{0}\cap T^{-n}A_{1}\right)>0\right\} \subseteq C.$$
    We prove our result by induction. Suppose for induction hypothesis we proved that $B\in p_{1}+p_{2}+\cdots+p_{k}$ for every $p_{1},p_{2},\ldots,p_{k}\in  E\left(J\left( \mathbb{N}\right)\right)$. 
    Assume that $p_{1},p_{2},\ldots,p_{k+1}\in E\left(K\left( J \mathbb{N}\right)\right)$. 
    To prove, $B\in p_{1}+p_{2}+\cdots+p_{k+1} $, it is sufficient to prove that 
    $$B\subseteq\left\{ m:-m+B\in p_{k+1}\right\}\in  p_{1}+p_{2}+\cdots+p_{k}.$$
   	 If $m\in B$, then $\mu\left(A_{0}\cap T^{-m}A_{1}\right)>0$. Let $D=A_{0}\cap T^{-m}A_{1}$ and $\mu \left(D\right)>0$. So, 
   	$\left\{ n:\mu\left(D\cap T^{-n}D\right)>0\right\} $ is dynamical $C^{\star}$-set, which implies, $E= \left\{ n:\mu\left(A_{0}\cap T^{-\left(n+m\right)}A_{1}\right)>0\right\} $ is a dynamical $C^{\star}$-set.
    Then for each $n\in E$, $$m+n\in B\implies n\in-m+B\in p_{k+1}.$$
   	 Which implies,  $B\subseteq\left\{ m:-m+B\in p_{k+1}\right\} \in p_{1}+p_{2}+\cdots+p_{k}$.
   	
\end{proof}
 Not all $C^{\star}$-sets enjoy the above property because there exist two minimal idempotents whose sum is not an idempotent by \cite[Exercise 6.1.4]{HS12}.
We know that a $IP^{\star}$-set is a $C^{\star}$-set, but the following theorem shows that a dynamical $IP^{\star}$-set may not be a dynamical $C^{\star}$-set.

\begin{theorem}
 	If $k\in\mathbb{N}$, then the set $k\mathbb{N}$ is a dynamical $IP^{\star}$-set. As a consequence $k\mathbb{N}$ is a dynamical $IP^{\star}$-set but not a dynamical $C^{\star}$ set.
 \end{theorem}
 
 \begin{proof}
 	We have to find a measure preserving system $\left(X, \mathcal{B}, \mu, T\right)$ and  $A\in\mathcal{B}$ with $\mu\left(A\right)>0$ such that $$\left\{ n\in\mathbb{N}:\mu\left(A\cap T^{-n}A\right)>0\right\} \subseteq k\mathbb{N}.$$
Take $X=\left\{1,2,\ldots,k\right\}$, $\mathcal{B}=\mathcal{P}\left(X\right)$, 
 $\mu\left(i\right)=\frac{1}{k}$ for all $i\in\left\{1,2,\ldots,k\right\}$ and $T$ is a permutation on $\left\{1,2,\ldots,k\right\}$ of $k$-cycle. Then 
 $$\left\{ n\in\mathbb{N}:\mu\left(\left\{1\right\}\cap T^{-n}\left\{1\right\}\right)>0\right\}= k\mathbb{N}.$$
 \end{proof} 

If $n\in\mathbb{N},$ and $A\subseteq \left(\mathbb{N},+\right)$ is a $C^{\star}$-set, then from \cite[Lemma 2.8]{HS20}, $nA$ is also a $C^{\star}$-set. As $d^\star (n\cdot A)<1,$  $n\cdot A$ is not a dynamical $C^{\star}$-set. But $n^{-1}A$ is a dynamical $C^{\star}$-set which can be deduced from the following Theorem.


\begin{theorem}\label{inverse}
 Let $A\subseteq\mathbb{N}$ is a dynamical $C^\star$-set
in $\left(\mathbb{N},+\right)$ and $n\in\mathbb{N}$ then $n^{-1}A$
is also dynamical $C^{*}$ set in $\left(\mathbb{N},+\right)$.
\end{theorem}

\begin{proof}
Let $\left(X,\mathcal{B},\mu, T\right)$
be a measure preserving system and $C_{0}$ and $C_{1}$ be two sets guaranteed by Definition \ref{D} such that $\left\{ m:\mu\left(C_{0}\cap T^{-m}C_{1}\right)>0\right\}\subseteq A$. $\left(X,\mathcal{B},\mu, T ^{n}\right)$ is a weak mixing system. Now 

\begin{align*}
    & m\in\left\{ m:\mu\left(C_{0}\cap T^{-nm}C_{1}\right)>0\right\}\\
    & \implies m\in\left\{ m:\mu\left(C_{0}\cap T^{-nm}C_{1}\right)>0\right\}\\
    & \implies m\in n^{-1}A.
\end{align*}
So $\left\{ m:\mu\left(C_{0}\cap T^{-nm}C_{1}\right)>0\right\} \subseteq n^{-1}A$
and this proves $n^{-1}A$ is dynamical C$^\star$-set.
\end{proof}

\section{Zigzag structure in dynamical $C^{\star}$-set }

In \cite{D14}, D. De introduced the notions of Almost minimal sequence.

\begin{definition}[\textbf{Almost minimal sequence}]
	A sequence $\langle x_{n}\rangle_{n=1}^{\infty}$ in $\mathbb{N}$
	is called minimal sequence if 
	\[
	\left(\bigcap_{m=1}^{\infty}cl\left(FS\left(\langle x_{n}\rangle_{n=m}^{\infty}\right)\right)\right)\cap J\left(\mathbb{N}\right)\neq\emptyset.
	\]
\end{definition}

In \cite[Theorem 2.7]{D14}, D. De represented a characterization of almost minimal sequences.

\begin{theorem}
     Let $\langle x_{n}\rangle_{n=1}^{\infty}$ be a sequence in $\mathbb{N}$. Then the following statements are equivalent:
    \begin{itemize}
    \item[(a)]  $\langle x_{n}\rangle_{n=1}^{\infty}$ is an almost minimal sequence;
        \item[(b)]  $FS\left(\langle x_{n}\rangle_{n=1}^{\infty}\right)$ is a $J$-set.
        \item[(c)] For all $m\in\mathbb{N}$, $FS\left(\langle x_{n}\rangle_{n=m}^{\infty}\right)$ is $C$-set. There exists an idempotent in  
        $$\left(\bigcap_{m=1}^{\infty}cl\left(FS\left(\langle x_{n}\rangle_{n=m}^{\infty}\right)\right)\right)\cap J\left(\mathbb{N}\right)\neq\emptyset.$$
    \end{itemize}
\end{theorem}

In \cite[Theorem 2.7]{D14}, D. De proved the following analogous version of Theorem \ref{ipstar}.

\begin{theorem}\label{c*}
	Let $\langle x_{n}\rangle_{n=1}^{\infty}$ be an almost minimal sequence in $\mathbb{N}$
	and $A\subseteq \mathbb{N}$ be a  C$^{\star}$-set in $\left(\mathbb{N},+\right)$. Then
	there exists a sum subsystem $FS\left(\langle y_{n}\rangle_{n=1}^{\infty}\right)$ of
	$FS\left(\langle x_{n}\rangle_{n=1}^{\infty}\right)$ such that $FS\left(\langle y_{n}\rangle_{n=1}^{\infty}\right)\cup FP\left(\langle y_{n}\rangle_{n=1}^{\infty}\right)\subseteq A$.
\end{theorem}

 At the end of this section, we prove the following analog version of Theorem \ref{zigzagipstar}for dynamical $C^{\star}$-set.
\begin{theorem}\label{zigzagCstar}
	Let $l\in\mathbb{N}$ and $A\subseteq\mathbb{N}$ be a dynamical $C^{\star}$ set in $\left(\mathbb{N},+\right)$. Then for any $l$ almost minimal sequences $\langle x_{n}^{\left(1\right)}\rangle_{n=1}^{\infty}$,$\ldots$,$\langle x_{n}^{\left(l\right)}\rangle_{n=1}^{\infty}$ in $\mathbb{N}$, the followings is true:
 \begin{enumerate}
     \item[(*)] for each $i\in\left\{ 1,2,\ldots,l\right\},$ there exists sum subsystems $\langle y_{n}^{\left(i\right)}\rangle_{n=1}^{\infty}$ of  $\langle x_{n}^{\left(i\right)}\rangle_{n=1}^{\infty}$   such that $$ZFS\left(\langle y_{n}^{\left(i\right)}\rangle_{i,n=1,1}^{l,\infty}\right)\bigcup ZFP\left(\langle y_{n}^{\left(i\right)}\rangle_{i,n=1,1}^{l,\infty}\right)\subset A.$$
 \end{enumerate}
\end{theorem}

As any dynamical $IP^{\star}$-set contains the combined zigzag structure of any finitely many sequences, it would be an interesting fact if we provide an existence of an example of $C^{\star}$-set, which is not a dynamical $IP^{\star}$-set. The following theorem satisfies our requirement.

\begin{theorem}\label{counter exam}
    There exists a dynamical $C^{\star}$-set, which is not an $IP^{\star}$-set.
\end{theorem}

Before proving the above theorem, we need to recall a result from \cite[Theorem 5.5]{KY07}, where R. Kung and X. Ye proved the following result on mild mixing systems.

\begin{theorem}
    A measure preserving system $\left(X,\mathcal{B},\mu, T\right)$ is mild  mixing if and only if for  $A_{0},A_{1}\in\mathcal{B}$, with $\mu\left(A_{0}\right)\mu\left(A_{1}\right)>0$, the set $N\left(A,B\right)= \left\{n:\mu\left(A_{0}\cap T^{-n}A_{1}\right)>0\right\}$ is an $IP^{\star}$-set.
\end{theorem}

\begin{proof}[\textbf{Proof of Theorem \ref{counter exam}.}]
    Let $\left(X, \mathcal{B}, \mu, T\right)$ be a weak mixing system,  which is not a mild mixing system. Then there exist $A,B\in\mathcal{B}$ such that $ N\left(A,B\right)=\left\{n:\mu\left(A_{0}\cap T^{-n}A_{1}\right)>0\right\}$ is not an $IP^{\star}$-set but $N\left(A,B\right)$ is a dynamical $C^{\star}$-set.
\end{proof}

The following lemma shows that dynamical $C^{\star}$-sets are partially inverse translation invariant.

\begin{lemma}\label{shift}
 Let B be a dynamical $C^{\star}$-set in $\left(\mathbb{N},+\right)$.
 Then it follows that there is a dynamical $C^{\star}$-set $C\subset B$ such that for
each $m\in C$, $-m+C$ is a dynamical $C^{\star}$-set
\end{lemma}

\begin{proof}
Let us consider a weak mixing system $\left(X,\mathcal{B},\mu , T \right)$ with
   $A_{0},A_{1}\in\mathcal{B}$ such that $\mu\left(A_{0}\right)\mu\left(A_{1}\right)>0$
and
$$\left\{ n\in S:\mu\left(A_{0}\cap T^{-n}A_{1}\right)>0\right\} \subseteq B.$$
Let 
\[
C=\left\{ n\in S:\mu\left(A_{0}\cap T^{-n}A_{1}\right)>0\right\} .
\]
To see that $C$ is as required, let $m\in C$ and let $D=A_{0}\cap T^{-m}A_{1}$.
We claim that
\[
\left\{ n\in S:\mu\left(D\cap T^{-n}D\right)>0\right\} \subseteq -m+C.
\]
Let $n\in S$ such that $\mu\left(D\cap T^{-n}D\right)>0$. Then

$$
	\begin{aligned}
		D\cap T^{-n}D & =A_{0}\cap T^{-m}A_{1}\cap T^{-n}\left(A_{0}\cap T^{-m}A_{1}\right)\\
  &  \subseteq A_{0}\cap T^{-m}\left(T^{-n}A_{1}\right)\\
		&=  A_{0}\cap T^{-(m+n)}A_{1}.
	\end{aligned}$$

So $m+n\in C$ and $n\in -m+C.$
\end{proof}

The following lemma shows that the intersection of finite numbers of dynamical
IP$^{\star}$-sets is dynamical IP$^{\star}$-set.

\begin{theorem}
    The intersection of finite numbers of dynamical
IP$^{\star}$-sets is dynamical IP$^{\star}$-set.
\end{theorem}
\begin{proof}

If $\left(X_{1},\mathcal{B}_{1},\mu_{1},T_{1}\right)$ and $\left(X_{2},\mathcal{B}_{2},\mu_{2},T_{2}\right)$ are two measure preserving systems. The product system, denoted by $\left(X_{1}\times X_{2} ,\mathcal{B}_{1}\times \mathcal{B}_{2} ,\mu_{1}\times \mu_{2},T_{1}\times T_{2}\right)$ consists of the product space endowed with the product measure, with $\mathcal{B}_{1}\times \mathcal{B}_{2}$ denoting the smallest $\sigma$-algebra on $X_{1}\times X_{2}$ including the products $A_{1}\times A_{2}$ of measurable sets $A_{1}\in \mathcal{B}_{1}$, $A_{2}\in \mathcal{B}_{2}$ with $\mu_{1}\times \mu_{2}\left(A_{1}\times A_{2}\right)=\mu_{1}\left(A_{1}\right)\mu_{2}\left(A_{2}\right)$ and with $T_{1}\times T_{2}\left(x_{1},x_{2}\right)=\left(T_{1}x_{1}\times T_{2}x_{2}\right)$.
\end{proof}

From \cite[Proposition 4.6]{F81}, we get the following:
\begin{lemma}\label{product of weak mixing}
	The product of two weak mixing systems is weak mixing.
\end{lemma}

The following Theorem shows that the intersection of finite numbers of dynamical
$C^{\star}$-sets is a dynamical $C^{\star}$-set.

\begin{theorem}\label{intersection}
Let $A,B\subseteq\mathbb{N}$ are two dynamical $C^\star$-sets in $\left(\mathbb{N},+\right)$
then $A\cap B$ is also a dynamical $C^\star$-set in $\left(\mathbb{N},+\right)$.
\end{theorem}

\begin{proof}
Let $\left(X,\mathcal{B},\mu, T\right)$
and $\left(Y,\mathcal{C},\nu, S\right)$
be two weak mixing systems and $C_{0},C_{1},D_{0},D_{1}$ be four sets guaranteed
by the  definition dynamical $C^{\star}$-set. Then 
\[
\left\{ n\in\mathbb{N}:\mu\left(C_{0}\cap T^{-n}C_{1}\right)>0\right\} \subseteq A
\]
and 
\[
\left\{ n\in\mathbb{N}:\nu\left(D_{0}\cap T^{-n}D_{1}\right)>0\right\} \subseteq B.
\]

By Lemma \ref{product of weak mixing},  $\left(X\times Y, \mathcal{B}\times\mathcal{C},\mu\times\nu, T\times S \right)$
be a weak mixing system. Then $\left(\mu\times \nu\right)\left(C_{0}\times D_{0}\right)=\mu\left(C_{0}\right)\cdot\nu\left(D_{0}\right)>0$ and $\left(\mu\times \nu\right)\left(C_{1}\times D_{1}\right)=\mu\left(C_{1}\right)\cdot\nu\left(D_{1}\right)>0$. Then 

$$
	\begin{aligned}
		& \left\{n:  \left(\mu\times\nu\right)\left(\left(C_{0}\times D_{0}\right)\cap\left(T\times S\right)^{-n}\left(C_{1}\times D_{1}\right)\right)>0\right\}\\
  & = \left\{n: \mu\left(C_{0}\cap T^{-n}C_{1}\right)\nu\left(D_{0}\cap T^{-n}D_{1}\right)>0  \right\}\\
  & = \left\{n: \mu\left(C_{0}\cap T^{-n}C_{1}\right)>0\right\}\cap \left\{n: \nu\left(D_{0}\cap T^{-n}D_{1}\right)>0\right\} \\
& \subseteq A\cap B.
	\end{aligned}$$
So $A\cap B$ is a dynamical C$^{*}$-set.
\end{proof}

Now we are in the position to prove the Theorem \ref{zigzagCstar}.

\begin{proof}[\textbf{Proof of Theorem \ref{zigzagCstar}}.]
 Let $A$ be a dynamical $C^{\star}$-set. Then from Lemma \ref{shift}  there exists a  dynamical $C^{\star}$-set $C$ such that $C\subseteq A$ and for $n\in C$, $-n+C$ is a dynamical $C^{\star}$-set. As for each $ i\in\left\{ 1,2,\ldots,l\right\}$,  $ \langle x_{n}^{\left(i\right)}\rangle_{n=1}^{\infty} $ is an almost minimal sequence, 
choose a sequence $ \langle H_{1}^{\left(i\right)}\rangle_{i=1}^{l}$ in $\mathcal{P}_{f}\left(\mathbb{N}\right)$ such that $ \sum_{t\in H_{1}^{\left(i\right)}}x_{t}^{\left(i\right)}=y_{1}^{\left(i\right)}\in C$. Assume for $ m\in\mathbb{N}$, we have a sequence $ \langle y_{n}^{\left(i\right)}\rangle_{n,i=1,1}^{m,l}$ in $\mathbb{N}$ and $ \langle H_{n}^{\left(i\right)}\rangle_{n,i=1,1}^{m,l}$ in $\mathcal{P}_{f}\left(\mathbb{N}\right)$ such that
	\begin{enumerate}
		\item For each $ j\in\left\{ 1,2,\ldots,m-1\right\}  $ and $ i\in\left\{ 1,2,\ldots,l\right\}  $ $\max H_{j}^{\left(i\right)}<\min H_{j+1}^{\left(i\right)}$.
		\item If for each $j\in\left\{ 1,2,\ldots,m\right\} $, $ y_{j}^{\left(i\right)}=\sum_{t\in H_{j}^{\left(i\right)}}x_{t}^{\left(i\right)},$ and
             \item $B=ZFS\left(\langle y_{j}^{\left(i\right)}\rangle_{j,i=1,1}^{m,l}\right)\bigcup ZFP\left(\langle y_{j}^{\left(i\right)}\rangle_{j,i=1,1}^{m,l}\right)\subset C.$
	\end{enumerate}
 Define $E_1=ZFS\left(\langle y_{j}^{\left(i\right)}\rangle_{j,i=1,1}^{m,l}\right),$ and $E_2=ZFP\left(\langle y_{j}^{\left(i\right)}\rangle_{j,i=1,1}^{m,l}\right).$
	
	Now we have, $D=C\cap\bigcap_{n\in E_{1}}\left(-n+C\right)\cap\bigcap_{n\in E_{2}}\left(n^{-1}C\right)$ is a dynamical $C^{\star}$-set  by  Lemma \ref{intersection}. For each $i\in \{1,\ldots ,l\},$ choose a sequence $\langle H_{m+1}^{\left(i\right)}\rangle_{i=1}^{l}$ such that $ \max H_{m}^{\left(i\right)}<\min H_{m+1}^{\left(i\right)}$ and $y_{m+1}^{\left(i\right)}=\sum_{t\in H_{m+1}^{\left(i\right)}}x_{t}^{\left(i\right)}\in D$. Now it is a routine exercise to verify that these choices of $ \langle y_{m+1}^{\left(i\right)}\rangle_{i=1}^{l}$ completes the induction hypothesis, which completes the proof.
	
\end{proof}

\bibliographystyle{plain}

 \end{document}